\documentclass[12]{amsart}

\usepackage{amssymb}
\newtheorem{theorem}{Theorem}

\newtheorem{corollary}[theorem]{Corollary}
\newtheorem{lemma}[theorem]{Lemma}

\newtheorem{question}[theorem]{Question}
\newtheorem{proposition}[theorem]{Proposition}

\begin{document}
\title{A note on set-theoretic solutions of the Yang-Baxter equation }
\author{ Agata Smoktunowicz}

\subjclass[2010]{Primary 16T25, 16N80, 16P90, 16N40, 16N20, 20F16, 20E22, 20F28, 81R50, 20B25, 20F38, 20B35, 20F16, 20F29. } 
\keywords{braces, braided group, the Young-Baxter Equation}

\date{\today}

\begin{abstract}
 This paper shows that every finite non-degenerate involutive  set theoretic  solution $(X,r)$ of the Yang-Baxter equation whose permutation group $\mathcal {G}(X,r)$ has cardinality which is a cube-free number  is a multipermutation solution. 
 Some properties of finite braces are also investigated (Theorems \ref{pierwsza}, \ref{dodatkowy}  and  \ref{nil}). 
  It is also shown that if $A$ is a left brace whose cardinality  is an odd number and  $(-a)\cdot b=-(a\cdot b)$ for all $a, b\in A$, then $A$ is a two-sided brace and hence a Jacobson radical ring. It is also observed that the semidirect product and the wreath product of braces of a finite multipermutation level is a brace of a finite multipermutation level.
\end{abstract}
\maketitle

\section{introduction}  {\it Circa} 2005, Wolfgang Rump introduced the notion of braces, a generalisation of Jacobson radical rings, as a tool for investigating solutions of the Yang-Baxter equation. We follow his paper \cite{mob} pp 128  for the definition of a left brace: 
  Let $A$ be an abelian group together with a left distributive multiplication, that is
  $a\cdot (b+c)=a\cdot b+a\cdot c$ 
 for all $a,b, c\in A$. We call $(A, +, \cdot)$ a {\em left brace} if the {\em circle operation} $a\circ b=a\cdot b+a+b$ 
 makes $A$ into a group.  
 As mentioned in \cite{mob}, pp 129, for a left brace $A$  the associativity of $A^{\circ }$ is easily seen to be equivalent to the equation
$(a\cdot b+a+b)\cdot c=a\cdot (b\cdot c)+a\cdot c+b\cdot c$.   The group $A^{\circ }$ will be called the {\em adjoint group} of a  left  brace $A$. 
 In 2012, Ced{\' o}, Jespers and Okni{\' n}ski \cite{cjo} expressed the definition of a left brace in terms of  operation $\circ $. In  their paper, the adjoint group $A^{\circ }$ is called the multiplicative group of the  left  brace $A$; their definition is equivalent to the above definition by Rump.  
 We recall the definition from \cite{cjo}: 
 a left brace is an abelian group $(A, +)$ with a multiplication $\circ $ such that $(A, \circ)$ is a group and 
\[a\circ (b+c)+a=a\circ b+a\circ c\]
 holds for all $a, b, c\in A$. 
 Another interesting structure related to the Yang-Baxter equation, the  braided group,  was introduced in 2000,  by Lu, Yan, Zhu \cite{lzy}. 
 In \cite{Tatyana},  Gateva-Ivanova  showed that left braces are in one-to-one correspondence with braided groups with an involutive braiding operator.
  
Recall that a set-theoretic solution of the Yang-Baxter equation is a pair $(X,r)$ where $X$ is a set and $ r(x,y)=(\sigma _{x}(y), \tau _{y}(x)),$
  for $x,y\in X$, is a bijective map such that 
\[(r\times id_{X})(id_{X}\times r)(r\times id_{X})=(id_{X}\times r)(r\times id_{X})(id_{X}\times r).\]
 A solution $(X,r)$ is non-degenerate if the maps $\sigma _{x}$ and $\tau _{x}$ are bijective for each $x\in X$, and $(X,r)$ is involutive if $r^{2}=id_{X\times X}$. 

{\bf Convention.} By a solution of the Yang-Baxter equation we will mean a non-degenerate, involutive set-theoretic solution of the Yang-Baxter equation.   

 Let $R$ be a left brace; then  the solution $(R,r)$ of the Yang-Baxter equation 
 associated to brace $R$ is defined in the following way:  for $x, y\in R$ define   $r(x,y)=(u,v)$, where  $u=x\cdot y+y,   v=z\cdot x+x$ and where $z$ is the inverse of $u=x\cdot y+y$ in the adjoint group $R^{\circ }$ of $R$, for $x, y\in R$. This solution is called the solution {\em associated with the left brace $R$} and  will be denoted as $(R,r)$.

 The notions of retract of a solution and  multipermutation solution were introduced by  Etingof, Schedler and Soloviev in \cite{etingof}. A multipermutation solution is a generalization of Lybashenko's permutation solution.
  A solution $(X,r)$ is called a multipermutation solution of level $m$ 
if $m$ is the smallest nonnegative integer that, after applying the operation of retraction $m$ times,  the obtained solution has cardinality $1$. If such $m$ exists, the solution is also called a {\em multipermutation solution}, that is a solution which has a finite multipermutation level (for a detailed definition, see \cite{Tatyana}, \cite{cjo}).  Some interesting related results can be found in   \cite{bcjo, etingof, retractable, cjo, tatyana6, v}; for example, it is known a finite solution $(X,r)$ is a multipermutation solution, provided that the permutation group  $\mathcal {G}(X,r)$ is abelian (Theorem $4.3$ \cite{cjo},  for the case  when $(X,r)$ is infinite see Theorem $7.1$ \cite{tatyana6}).

Let $A$ be either a left brace or a right brace. We define  $A^{(n)}=A^{(n-1)}\cdot A$ and $A^{n}=A\cdot A^{n}$, where $A^{(1)}=1$, and say that a left brace $A$ has  {\em a finite multipermutation level} if   $A^{(n)}=0$ for some $n$ or equivalently the solution of the Yang-Baxter equation associated to $A$  is a multipermutation solution (for some related results see \cite{GIS, rump, smok7}).

{\bf Remark} In \cite{rump}, Rump  introduced radical chains $A^{n}$ and $A^{(n)}$ for a right brace $A$. 
 Rump showed that, if $A$ is a right brace, then $A^{n}$ is a two-sided ideal of $A$, and $A^{(n)}$ doesn't need to be a two-sided ideal of $A$. 
 Notice that if $A$ is a left brace then $A^{(n)}$ is a two-sided ideal of $A$, and $A^{n}$ doesn't need to be a two-sided ideal of $A$. 

A semidirect product of braces and wreath product of braces  is an interesting construction. The semidirect product of braces was introduced by Rump in \cite{r36}. The  wreath product of braces  was investigated in 2008 in Corollaries $3.5$ and $3.6$  \cite{cjr1}, \cite{cjr} and in Corollary $6.1$ \cite{cjo}.  The wreath product of solutions  was  studied in 2009 in  \cite{ta1} (see Theorem $8.7$). 
  We have the following observation, which is related to Theorem $8.7$ from \cite{ta1}. 

\begin{proposition}\label{100} Let $A$ and $B$ be left braces of a finite multipermutation level. Then 
the semidirect product $A\rtimes B$ and  the  wreath product $A\wr B$  of braces $A$ and $B$ is a brace of a finite multipermutation level. 
\end{proposition}
  
 Braided groups and braces have interesting connections with group theory.
 In \cite{cjr}, it was shown that every finite nilpotent group can be embedded into an adjoint group of a finite brace whose adjoint group is nilpotent.
 In  Corollary $6.1$ \cite{cjo}, it was shown that every finite solvable group is a subgroup of an adjoint group of a finite brace. We notice that, by using the same proof as in Corollary $6.1$ in \cite{cjo} and  Proposition \ref{100},  it is possible to show a slightly stronger result, namely:

{\bf Remark} (Related to Corollary $6.1$, \cite{cjo})  Let $G$ be a finite solvable group. There is a finite left brace $A$ of a finite multpermutation level, such that $G$ is  a  subgroup of the adjoint group $A^{\circ }$ of $A$.

 We don't know the answer to the following questions.
\begin{question}   
 Let $G$ be a finite group which is the adjoint group of some left brace $A$. Does it follow that $G$ is the adjoint group 
 of some left brace of a finite multipermutation level?
\end{question} 
 It was shown by Rump that two-sided braces are exactly Jacobson radical rings \cite{mob} pp. 129. Therefore, every finite two-sided brace is a nilpotent ring and it has a finite mulitpermutation level ( see \cite{mob} pp. 135,  \cite{rump} pp. 154, and  for a detailed proof see  \cite{cjo1}).
 By Corollary of \cite{K}, for every prime $p$, a group of order $p^{4}$  is the adjoint group of a two-sided
brace if and only if it is abelian or has class $2$. By Theorem 2.1 of \cite{cjo3},
there exists a group of order $p^{4}$ of class $3$  which is the adjoint group of a left
 brace, but is not the adjoint group of a two-sided brace, for example 
\[\langle x,y |\text{ }  x^{3}=y^{3}=[x,y]^{3}=[x,[x,y]]^{3}=[y,[x,y]]=1, \text{and }  [x,[x,y]] \text{ is central }\rangle\]
has order $3^{4}$ and it has class $3$ and it is the adjoint group of a left brace.

 Nil rings are examples of left braces and two-sided braces.  The following question are related to Eggert's conjecture \cite{4}.
\begin{question} (Amberg, Kazarin, Sysak \cite{4,6}). If $R$ is a nil ring whose adjoint
group $R^{\circ }$ is finitely generated, is $R$ nilpotent?
\end{question}
\begin{question}  Let $F$ be a field of characteristic not two.  An associative $F$-algebra
$R$ gives rise to the commutator Lie algebra $R^{-}=(R, [a,b]=ab-ba)$.
 If $R$  is a nil algebra such that $R^{-}$ is finitely generated Lie algebra, is $R$
nilpotent?
\end{question}
 In \cite{jain} Alahmedi, Alsulami, Jain and Zelmanov give 
sufficient conditions for the Lie algebra $R^{-}$ be finitely generated.  
 In \cite{agv} Angiono, Galindo and Vendramin provided Lie-theoretical analogs of braces.
Some interesting results on nil and nilpotent subsets   in  Lie rings can be found in  \cite{gb, mz, p1,  psz, sz}.

We will use the following notation.
\[a^{\circ( n)}=a\circ a\circ \cdots \circ a\] where $a$ appears $n$ times, so $a^{\circ (n)}$ is the $n$-th power of $a$ in the adjoint group $A^{\circ }$ of $A$.
 The main result of this paper is the following. 
\begin{theorem}\label{pierwsza}
 Let $(A, +, \cdot)$ be a left brace and let $A_{p}$, $A_{q}$ be the Sylow's subgroups of the additive group of $A$ of cardinalities respectively $p^{n}$ and $q^{m}$ for some prime numbers $q$ and $p$ and some natural numbers $m,n$. 
 Then the following holds:
\begin{itemize}
\item [1.]
If $p$ doesn't divide $q^{t}-1$ for any $1\leq t\leq m$, then \[A_{p}\cdot A_{q}=0.\]
\item[2.]  Let $a\in A_{p}, b\in A_{q}$ and $k$ be the maximal number such that $p^{k}$ divides $q^{t}-1$ for some  
$1\leq t\leq m$. Then $a^{\circ (p^{k})}\cdot b=0$. 
\end{itemize}
\end{theorem}
 
 Notice that Theorem \ref{pierwsza} implies that if $q<p$ and $p$ doesn't divide $q^{t}-1$ for any $0<t\leq m$, then every left brace of order $p^{n}q^{m}$ is not a simple left brace, as $A_{p}$ is an ideal in $A$. Moreover, if $A_{p}$ has a nonzero socle, then $A$ has a nonzero socle.
 Note that in \cite{B}  Bachiller  gave the first examples of
non-trivial finite simple left braces.  
 We would like to pose a related open question. 
\begin{question}\label{5} 
 Let  $p,q$ be prime numbers and $\alpha, \beta $ be positive integers. Assume that  $p$ divides $q^{\beta }-1$ and 
  $q$ divides $p^{\alpha }-1$. Is there a simple brace of cardinality $p^{\alpha }q^{\beta }$?
\end{question} 
 Our next result follows from Theorem \ref{pierwsza}
\begin{corollary}\label{dodatkowy}
Let $A$ be a left brace  of cardinality $p_{1}^{\alpha (1)}\ldots p_{n}^{\alpha (n)}$ for some $n$, some prime numbers $p_{1}<p_{2}< \ldots < p_{n}$ and some positive integers $\alpha (1), \ldots , \alpha (n)$. 
  Let $A_{i}$ denote the Sylow's subgroup of the additive group of $A$ of cardinality $p_{i}^{\alpha (i)}$.
 Suppose that for some $m\leq n$ the brace $A_{m}$ has a nonzero socle and  $p_{m}$ doesn't divide  
  \[p_{j}^{i}-1\]   for all $j\leq n$ and each  $i\leq \alpha (j)$.  Then the socle of $A$ is non-zero.
\end{corollary}

 Recall that the socle of a left brace is defined as 
Soc$(A)=\{a\in A: a\cdot b=0$ for all $b\in A\}$ (see \cite{rump}, and for left braces \cite{cjo}). 
 We have the following corollary of Corollary \ref{dodatkowy}.
\begin{corollary}\label{49} 
 Let  $A$ be a left brace of cardinality $p_{1}^{\alpha _{1}}\cdots p_{n}^{\alpha _{n}}$ for some pairwise distinct prime numbers $p_{1}, p_{2},\ldots ,p_{n}$  and some positive integers $\alpha _{1}\ldots \alpha _{n}$. Assume that for every $i\leq n$, and every $k<i$, 
 $p_{i}$ doesn't divide  \[p_{k}^{t}-1\] for any $t\leq \alpha _{k}$. If all the Sylow's subgroups of the additive group of $A_{m}$ are braces of a finite multipermutation level then $A$ is a brace of a finite multipermutation level. 
\end{corollary}
 Recall that if $A$ is a finite left brace then all the Sylow's subgroups of the additive group of $A$ are also left braces (see Lemma $17$ \cite{smok7}).

As an application of Theorem \ref{pierwsza}, we obtain the following corollary.
\begin{corollary}\label{momik}
 If $A$ is a brace whose cardinality is a cube-free number, then the socle of $A$ is nonzero. Moreover,  $A$ is a brace of a finite multipermutation level.
\end{corollary}
 We can formulate Corollary \ref{momik} in the language of braided groups as follows. 
\begin{corollary}\label{momik2}
 If $(G, \sigma )$ is a symmetric group (in the sense of Takeuchi)  whose cardinality is  a cube-free number, 
 then $(G, \sigma)$ has a finite multipermutation level.
\end{corollary} 
 Our next result is the following.
 \begin{theorem}\label{momik3}
 Every finite solution $(X,r)$ of the Yang-Baxter equation whose permutation group $\mathcal {G}(X,r)$ has cardinality which is a cube-free number  is a multipermutation solution.
\end{theorem}
 Recall that the permutation group $\mathcal{G}(X,r)$ of left actions associated with $(X,r)$ was introduced by Gateva-Ivanova in \cite{tatyana2} (see also \cite{Tatyana}). 

 Our last application of  Theorem \ref{pierwsza} is the following. 
\begin{corollary}\label{ostat}
 Let $A$ be a finite left brace and let $A_{1}, \ldots , A_{n}$ be the Sylow's subgroups of the additive group of $A$. If for every $i$, $A_{i}\cdot A_{i}=0$ and the  additive group of $A_{i}$ is cyclic, then $A$ 
 is a left brace of a finite multipermutation level.
\end{corollary}

 Let $A$ be a left brace and $a\in A$, then we denote $a^{1}=a$ and for $i>1$ we denote  $a^{i+1}=a\cdot a^{i}$.
  Our next result is related to Theorem $1$ from \cite{smok7}.
\begin{theorem}\label{nil}
 Let $A$ be a finite left brace; then the adjoint group $A^{\circ }$ is nilpotent if and only if there is a number $m$ such that $a^{m}=0$ for every $a\in A$. Moreover, if $a,b\in A$ are elements of distinct Sylow's subgroups of the additive
 group of $A$, then $(a+b)^{n}=0$ for some $n$ implies $a\cdot b=b\cdot a=0$. 
\end{theorem}
 In \cite{gvend} Guarnieri and Vendramin posed several interesting conjectures. Conjecture $6.4$ states that if $p<q$  are prime numbers and $p$ doesn't divide $q-1$, then the number of not isomorphic  left braces of order $p^{2}q$ is $4$; notice that then $q$ doesn't divide $p^{2}-1=(p-1)(p+1)$; so if $A_{p}$, $A_{q}$ are the Sylow's subgroup of the additive group of  a brace $A$ of such cardinality $p^{2}q$ then $A_{p}\cdot A_{q}=0$ and $A_{q}\cdot A_{p}=0$ by Theorem \ref{pierwsza}. Therefore, $A$ is a direct sum of braces of order $p^{2}$ and $q$, and the truth of Conjecture $6.4$ follows from \cite{B3}, where it was shown that there are exactly $4$ nonisomorphic braces of order $p^{2}$.
 We notice that since groups of some cardinalities are nilpotent it is also possible to use Theorem $1$  from \cite{smok7}, which asserts that a brace whose adjoint group is nilpotent is a direct sum of the braces which are  Sylow's subgroups of its additive group. 

Let $k$ be a field. In analogy with $k$-algebras,  Catino and Rizzo \cite{cr} introduced braces whose additive groups are $k$-vector spaces such that $(\alpha a)\cdot b=\alpha (a\cdot b)$, for $\alpha \in K$. Rump called such braces $k-linear$ or $k$-braces; Catino and Rizzo called them circle algebras.
  In our next result we consider a similar property  for left braces.
\begin{theorem}\label{sardegna}
 Let $A$ be a left brace whose additive group has no elements of order $2$.  
 If $(-a)\cdot b=-(a\cdot b)$ for every $a,b\in A$, then $A$ is a two-sided brace, and hence a Jacobson radical ring. Furthermore, if $A$ is finite, then $A$ is a nilpotent ring.
\end{theorem} 
 For  some related  results see the references \cite{jain}--\cite{v}.
\section{Notation }
Let $(A, +, \cdot )$ be a left brace defined as at the beginning of this paper, and  let $\circ $ be the operation such that $a\circ b= a\cdot b+a+b$. We will use the following notation.
\[a^{\circ( n)}=a\circ a\circ \cdots \circ a\] where $a$ appears $n$ times, and
\[a^{n}=a\cdot (a\cdot (\cdots (a))),\] where $a$ appears $n$ times. Let $a,b\in A$, we
 inductively define elements $e_{n}=e_{n}(a,b)$  as 
$e_{0}=b$,  $e_{1}=a\cdot b$ and  \[e_{i+1}=a\cdot e_{i}.\]
If $n$ is a natural number and $a\in A$ then $n\cdot a$ denotes the sum of $n$ copies of $a$, and $(-n)\cdot a$ denotes the sum of 
 $n$ copies of elements $-a$. 

 Given a prime number $q$, let $F_{q}$ denote the field of $q$ elements, and let $F_{q}[x]$ denote the polynomial ring over $F_{q}$ in one variable $x$.

\section{Multipermutation solutions}
 Let $A$ be a left brace, and let $a,b\in A$. Recall that we inductively define elements $e_{n}=e_{n}(a,b)$  as 
$e_{0}=b$,  $e_{1}=a\cdot b$ and  $e_{i+1}=a\cdot e_{i}.$
We will start with the following supporting lemma:
\begin{lemma}\label{21}
 Let $A$ be a left brace and let $a, b\in A$ and $n$ be a positive integer. 
  Let notation be as above. 
Then 
\[(a^{\circ (n)})\cdot b=\sum_{i=1}^{n}{n\choose i}\cdot e_{i}.\]
 Moreover,
\[a^{\circ (n)}=\sum_{i=1}^{n}{n\choose i}\cdot a^{i}.\]
\end{lemma}
\begin{proof}
 Observe that it can be shown by induction on $n$ that $a^{\circ (n)}=a\circ ( a \circ (\ldots \circ (a\circ a)))=\sum_{i=1}^{n}{n\choose i}\cdot a^{i}$, as $a^{\circ (n+1)}=a\circ a^{\circ (n)}=a+a^{\circ (n)}+a\cdot a^{\circ (n)}=a+(\sum_{i=1}^{n}{n\choose i}\cdot a^{i})+(\sum_{i=1}^{n}{n\choose i}\cdot a^{i+1})=\sum_{i=1}^{n+1}{n+1\choose i}\cdot a^{i}$.

 It remains now to show that 
 $ \sum_{i=1}^{n}{n\choose i}\cdot e_{i}=(a^{\circ (n)})\cdot b$. 
 We know that $\circ $ is an  associative operation, so
$a^{\circ (n)}\circ b=a\circ (a\circ ( \cdots \circ (a\circ b))).$
Observe that $(a^{\circ (n)})\cdot b=(a^{\circ (n)})\circ b-(a^{\circ (n)})- b$,
so it suffices to show that 
\[\sum_{i=1}^{n}{n\choose i}\cdot e_{i}+a^{\circ (n)}+ b=a^{\circ (n)}\circ b.\]
We use the induction on $n$, for $n=1$ we have 
$e_{1}+a+b=a\circ b$ so the result holds. Suppose that the result holds for some $n$.
We  multiply by $a$ from the left to get
\[a\circ (\sum_{i=1}^{n}{n\choose i}\cdot e_{i}+a^{\circ (n)}+ b)=a^{\circ (n+1)}\circ b.\]
 Observe that the left hand side equals
$a\cdot (\sum_{i=1}^{n}{n\choose i}\cdot e_{i}+a^{\circ n}+ b)+a+ (\sum_{i=1}^{n}{n\choose i}\cdot e_{i}+a^{\circ n}+ b)=\sum_{i=1}^{n}{n\choose i}\cdot (e_{i}+e_{i+1})+a\cdot a^{\circ n}+a\cdot b+a+a^{\circ n}+ b=
\sum_{i=1}^{n+1}{(n+1)\choose i}\cdot e_{i}+a^{\circ (n+1)}+b,$ 
 which finishes the inductive argument
\end{proof}
\begin{lemma}\label{10}
Let $A$ be a finite left  brace, and $a,b\in A,$ $a, b\neq 0$, and let $p, j$ be natural numbers. Then $ a^{\circ (p^j)}\cdot b=0$  if and only if  $\sum_{i=1}^{p^{j}}{p^{j}\choose i}\cdot e_{i}=0.$
\end{lemma}
\begin{proof} It follows because 
 $a^{\circ (p^{j})}\cdot b=\sum_{i=1}^{p^{j}}{p^{j}\choose i}\cdot e_{i}$  by Lemma \ref{21}.
\end{proof}
\begin{lemma}\label{aab}
 Let $A$ be a left brace and let $a,b\in A$. Suppose that $a^{\circ (p^{j})}=0$ and $q^{m}\cdot b=0$ for some distinct prime numbers $p,q$ and some positive integers $m,j$.
 If $a\cdot (a\cdot b)=0$ then $a\cdot b=0$.
\end{lemma}
\begin{proof}  By Lemma \ref{21} we have 
$0=a^{\circ (n)}\cdot b=\sum_{i=1}^{n}{n\choose i}\cdot e_{i}$ 
where $n=p^{j}$ and  $e_{i}$ are defined as in Lemma \ref{21}. Observe that $e_{i}=0$ for all $i>1$ as $e_{2}=a\cdot (a\cdot b)=0$.
 Therefore $n\cdot e_{1}=p^{j}\cdot e_{1}=0$, and since $p^{j}$ and $q$ are co-prime then $e_{1}=0$, so $a\cdot b=e_{1}=0$.   
\end{proof}

 For a prime number $q$, let $F_{q}$ denote the field of $q$ elements and $F_{q}[x]$ be the polynomial ring in one variable over $F_{q}$.
 Let $\mathbb Z$ be the ring of integers and  ${\mathbb Z}[x]$ be the polynomial ring in one variable over $\mathbb Z$.  
\begin{lemma}\label{polynomials1}
Let $A$ be a left brace and $a,b\in A,$ $a, b\neq 0$ and let $e_{i}=e_{i}(a,b)$ for every $i$.   
 Let $f(x), g(x), h(x)\in {\mathbb Z}[x]$ be such that $h(x)=g(x)f(x)$ where $f(x)=\sum_{i=0}^{k}j_{i}x^{i}$,  $g(x)=\sum_{i=0}^{l}l_{i}x^{i}$  for some natural numbers $k,l$.
 Denote $f=\sum_{i=0}^{k}j_{i}\cdot e_{i}$, $g=\sum_{i=0}^{l}l_{i}e_{i}$. If $f=0$ then  $h=0$.
\end{lemma}
\begin{proof} Observe that since $f=0$ then $0=a\cdot f=\sum_{i=0}^{k}j_{i}\cdot e_{i+1}$; notice that the element $a\cdot f$ corresponds to the polynomial $x\cdot f(x)$. Similarly the  element  $0=a\cdot (a\cdot f)=\sum_{i=0}^{k}j_{i}\cdot e_{i+2}$ corresponds to the polynomial $x^{2}f(x)$. Continuing in this way we get that $h=0$, since it corresponds to the 
polynomial $g(x)f(x)$.
\end{proof}
\begin{lemma}\label{polynomials} 
Let $A$ be a left brace and $a,b\in A,$ $a, b\neq 0$. Suppose that $q\cdot b=0$ for some natural number $q$.
 Let $f(x), g(x), p(x), q(x), h(x)\in {\mathbb Z}[x]$ and let  $r(x)=p(x)f(x)+q(x)g(x)+q\cdot h(x)$ where $f(x)=\sum_{i=0}^{k}j_{i}x^{i}$,  $g(x)=\sum_{i=0}^{l}l_{i}x^{i}$ and $h(x)=\sum_{i=0}^{l'}h_{i}x^{i}$, $r(x)=\sum_{i=0}^{m}r_{i}x^{i}$, for some natural numbers $k,l,l', m$.
 Denote $f=\sum_{i=0}^{k}j_{i}\cdot e_{i}$, $g=\sum_{i=0}^{l}l_{i}e_{i}$ and  $r=\sum_{i=0}^{m}r_{i}e_{i}$. If $f=0$ and $g=0$ then $r=0$.
\end{lemma}
 \begin{proof} Observe first that $q\cdot b=0$ implies $q\cdot e_{i}=0$ for all $i$. Therefore, $t=q\cdot   \sum_{i=1}^{l'}h_{i}e_{i}=0$, observe that $t$ corresponds to the polynomial $q\cdot h(x)$. Let $p=\sum_{i}p_{i}e_{i}$ and $q=\sum_{i}q_{i}e_{i}$ be elements corresponding to the polynomials 
 $p(x)f(x)$ and $q(x)g(x)$; by Lemma \ref{polynomials1} we get $p=q=0$. Observe that $r=p+q+t$, and so $r=0$. 
\end{proof}

\begin{theorem}\label{25} 
Let $A$ be a finite left  brace and $a,b\in A,$ $a\cdot b\neq 0$. Suppose that $a^{\circ (p^{j })}=0$ and $q\cdot b=0$ for some distinct prime numbers $p, q$ and some natural  number $j$. Let $A_{q}$ be the Sylow's subgroup of the additive group of $A$, then $A_{q}$ has cardinality $q^{m}$ for some $m$, and $b\in A_{q}$. Let $k$ be the maximal number such that $p^{k}$ divides $q^{i}-1$ for some $i\leq m$.
 Then \[a^{\circ (p^{k})}\cdot b=0.\] 
 In particular $a\cdot b=0$ if $p$ doesn't divide $q^{i}-1$ for any $i\leq m$. 
\end{theorem}
\begin{proof} By a slight abuse of notation we can consider natural numbers smaller than $q$ as elements of the field $F_{q}$.  Denote $e_{0}=b, e_{1}=a\cdot b$ and inductively $e_{i+1}=a\cdot e_{i}$.
 Denote $f(x)=(x+1)^{p^{j}}-1$, then $f(x)=\sum_{i=1}^{p^j}{p^{j}\choose i}x^{i}.$  Consider elements $\sum_{i=0}^{m}n_{i}\cdot e_{i}\in A_{q}$, where $0\leq n_{0}, n_{1}, \ldots, n_{m}< q$ then there are $q^{m+1}$ elements in this set, therefore some two of them are equal.   Notice that $q\cdot b=0$ implies $q\cdot e_{i}=0$ for every $i$.
Therefore, $\sum_{i=0}^{m}n'_{i}\cdot e_{i}=0$, for some $0\leq n'_{0}, n'_{1}, \ldots, n'_{m}< q$ (not all equal to zero).
 Denote $g(x)= \sum_{i=0}^{m}n'_{i}\cdot x^{i}$. It is known that an irreducible polynomial 
  of degree $\alpha $ from $F_{q}[x]$ divides the polynomial $x^{q^{\alpha }}-x$. Therefore, a polynomial of degree not exceeding $m$ from $F_{q}[x]$ divides the  polynomial $x^{m}(\prod _{i=1}^{m}(x^{q^{i} -1}-1))^{m}$ in $F_{q}[x].$  
 Consequently  $h(x+1)=g(x)$ divides  the polynomial $(x+1)^{m}(\prod _{i=1}^{m}((x+1)^{q^{i} -1}-1))^{m}$ in $F_{q}[x]$. Denote $l(x)=  (x+1)^{m}(\prod _{i=1}^{m}((x+1)^{q^{i} -1}-1))^{m}$ and let $l(x)=\sum_{i}\alpha _{i}x^{i}$ for some $0\leq \alpha _{i}<q$, then  
  $\sum_{i}\alpha _{i}e_{i}=0$ by Lemma \ref{polynomials}.

 Let $t(x)=\sum _{i=0}^{t}j_{i}\cdot x^{i}$ be the greatest common divisor of $f(x)$ and $l(x)$ in $F_{q}[x]$; then there exist polynomials $h(x), p(x), q(x)\in {\mathbb Z}[x]$ such that
\[t(x)=p(x)\cdot f(x)+q(x)\cdot l(x)+q\cdot h(x).\] 
 By Lemma \ref{polynomials} we get that 
$\sum_{i=0}^{l}j_{i}e_{i}=0$, provided that $f=0$, where in our case $f=\sum_{i=1}^{p^{j}}{p^{j}\choose i}\cdot e_{i}$.
  Recall that $a^{\circ (p^{j })}=0$, by Lemma \ref{10} we get \[0=\sum_{i=1}^{p^{j}}{p^{j}\choose i}\cdot e_{i}= f,\] hence $\sum_{i=1}^{l}j_{i}e_{i}=0$, as required.
 
 We will show now that $t(x)=(x+1)^{p^{k}}-1$. The greatest common divisor of $f(x)=(x+1)^{p^{j}}-1$ and 
$(x+1)^{q^i}-(x+1)=((x+1)^{q^{i}-1}-1)(x+1)$ equals $(x+1)^{p^{s}}-1$ for some $s$, since $(x+1)$ doesn't divide $f(x)$. 
Observe that $f(x)$ doesn't have multiple roots in any field extension of $F_{q}$ because $f(x)$ and  $f'(x)=p^{k}(x+1)^{p^{k}-1}$ have no common roots.
 Since $f(x)$ has no multiple roots then $t(x)$ doesn't have multiple roots. Notice also that $(x+1)^{p^{i}}-1$ divides polynomial $(x+1)^{p^{i+1}}-1$.  It follows that $t(x)=(x+1)^{p^{k}}-1$.
  By Lemma  \ref{10} we get that $a^{\circ (p^{k})}\cdot b=0$.

 If $p$ doesn't divide $q^{i}-1$ for any $i\leq k$ then $t(x)=(x+1)-1=x$ hence $e_{1}=a\cdot b=0$,  as required.
\end{proof} 
{\bf Proof Theorem \ref{pierwsza}.} Observe that part [1] follows from part [2] applied for $k=0$. Therefore, we will prove [2]. Let $a'=a^{\circ (p^{k})}$, and if $k=0$ then $a'=a$.  Suppose on the contrary that $a'\cdot b'\neq 0$ for some $b'\in A_{q}$. Let $\alpha $ be such that $a'\cdot (q^{\alpha }\cdot b')=0$ and $a'\cdot (q^{\alpha -1}\cdot b')\neq 0.$ Denote $b=a'\cdot (q^{\alpha -1}\cdot b')\neq 0.$
 Observe that $q\cdot b=0$ and $a'\cdot b=  a'\cdot (a'\cdot (q^{\alpha -1}\cdot b'))\neq 0$ 
 by Lemma \ref{aab}. Therefore $0\neq a'\cdot b=a^{\circ (p^{k})}\cdot b$.  This is impossible by Theorem \ref{25}.

{\bf Proof of Corollary \ref{dodatkowy}.}
  By Theorem \ref{pierwsza} we get 
 $A_{m}\cdot A_{j}=0$ 
for all $j\leq n$, since $p_{m}$ doesn't divide $p_{j}^{i}-1$ for each $i\leq \alpha (j)$.
 Let $a$ be in the socle of  the brace $A_{n}$, then $a\cdot A_{i}=0$ for all $i$, it follows that $a$ is in the socle of $A$.

{\bf Proof of Corollary \ref{49}.} Observe that if a  brace $A$ satisfies the assumptions of Corollary \ref{dodatkowy} then the retraction of $A$ also satisfies the assumptions of Corollary \ref{dodatkowy}. Recall that in \cite{rump}  Rump  has shown that the retraction of a brace $A$ is isomorphic to $A/Soc(A)$. The result now follows from Corollary \ref{dodatkowy} applied several times.

{\bf Remark}. The author is grateful to Leandro Vendramin for providing a list of braces which small cardinalities. In particular it follows from this list that all braces with cardinalities $6, 8, 12, 36$ have a finite multipermutation level.

\begin{theorem}\label{187}
 Let $A$ be a finite left brace whose cardinality is a  cube-free number; then $A$ has a nonzero socle. 
\end{theorem}
\begin{proof}  Let $A=\sum_{i=1}^{n}A_{i}$ where $A_{i}$ are Sylow's subgroups of the additive group of $A$. 
 Let $A_{i}$ have cardinality $p_{i}^{\alpha (i)}$ for $i=1,2, \ldots , n$ where $p_{1}<p_{2}< \ldots < p_{n}$ are  prime numbers.
 It is known that every $A_{i}$ is a brace (see for example Lemma $17$ in \cite{smok7}). It is known that all groups of order $p$ and $p^{2}$ are abelian (see \cite{B3} for some related results). Terefore, $A_{n}$ is a  two-sided brace, it follows that the socle of $A_{n}$ is nonzero.
 Therefore there is an  element $a\in A_{n}$ such that $a\cdot b=0$ for all $b\in A_{n}$. Assume first that $p_{n}>3.$
We will show that $a\cdot b=0$ for every $b\in A$. By Theorem \ref{pierwsza} we get  
 $A_{n}\cdot A_{i}=0$ 
for all $i<n$, since $p_{n}$ doesn't divide  $p_{i}-1$ nor $p_{i}^{2}-1=(p_{i}-1)(p_{i}+1)$.
 Therefore $a\cdot A=0$, as required.
 Assume now that $A_{n}$ has cardinality  equal to either $3$ or $9$, and $A_{n-1}$ has cardinality $2$ or $4$. It follows that $A$ is a brace of one of the following cardinalities: $6, 18, 12, 36$. By the remark above this theorem it follows that $A$ has a nonzero socle.
\end{proof}
{\bf Proof of Corollary \ref{momik}.} It follows from Theorem \ref{187} applied several times.

 {\bf Proof of Corollary \ref{momik2}.}
 In \cite{Tatyana}  Gateva-Ivanova  showed that left braces and braided groups with involutive braiding operators are in one-to-one correspondence. Using this correspondence we see that  
Corollary \ref{momik2} follows from Corollary \ref{momik}.
 
{\bf Proof of Theorem \ref{momik3}.}
In \cite{Tatyana} Gateva-Ivanova showed that a solution $(X,r)$ is a multipermutation solution if and only if 
 the symmetric group $\mathcal {G}(X,r)$ has a finite multipermutation level. It is known that $\mathcal {G}(X,r)$ has a structure of a left brace, and 
Theorem \ref{momik3} follows from  Corollary \ref{momik}.

{\bf Proof of Corollary \ref{ostat}.}  Let the cardinality of $A_{i}$ be $p_{i}^{\alpha (i) }$
 where $p_{i}$ is prime. Let $A_{j}$ have the largest cardinality among braces $A_{1}, \ldots , A_{n}$. It follows that $p_{j}^{\alpha (j) }$ doesn't divide 
 $p_{t}^{i}-1$ for any $t$ and any  $i\leq \alpha (t).$ 
 Let $a$ be a generator of the additive group of $A_{j}$, then $a^{\circ ({p_{j}}^{\alpha (j)-1})}\neq 0$ and $a^{\circ ({p_{j}}^{\alpha (j)})}=0$, this follows because $A_{j}\cdot A_{j}=0$. 
By Theorem \ref{pierwsza} [2] there is $k\leq \alpha (j)-1$ such that $a^{\circ ({p_{j}}^{k})}\cdot b=0$ for every $b\in A_{i}$ for $i\neq j$. Observe that 
$a^{\circ ({p_{j}}^{k})}\cdot b=0$ for $b\in A_{j}$ because $A_{j}\cdot A_{j}=0$.
 Therefore $a^{\circ ({p_{j}}^{k})}$ is in the socle of $A$. Observe that brace $A/Soc(A)$ satisfies the assumptions of this theorem, so it has a nonzero socle. Continuing in this way we get that $A$ has a finite multipermutation level (because the retraction of a brace $A$ equals $A/Soc(A)$ by a result of Rump \cite{rump}).

{\bf Proof of Theorem \ref{nil}.}  Observe that  $(a+b)^{n}=e_{n-1}(a+b,a)+e_{n-1}(a+b,b)$, where elements  $e_{n-1}(a+b,a)$ and $e_{n-1}(a+b,b)$ are defined as in Section $2$. Therefore $(a+b)^{n}=0$ implies  $e_{n-1}(a+b,a)=e_{n-1}(a+b,b)=0$, since $e_{n-1}(a+b,a)$ and $e_{n-1}(a+b,b)$ are from different Sylow's subgroups of the additive group of $A$. 
We can assume that $p^{m}\cdot a=0$,  $q^{m'}\cdot b=0$ for some distinct prime numbers $p,q$ and some natural numbers $m,m'$. 
We can assume that $m>n$, if necessary taking bigger $m$.  By Lemma \ref{21}
\[((a+b)^{\circ( j)})\cdot b=\sum_{i=1}^{j}{j\choose i}\cdot e_{i}(a+b,b).\]  Let $\alpha =q^{t}$ where $t>m'$ be such that $p^{m+1}$ divides $q^{t}-1$ (we know that $p$ divides $q^{p-1}-1$ so $p^{m+1}$ divides $q^{(p-1)p^{m}}-1$).
 It follows that ${\alpha \choose i}$ is divisible by $p^{m}\cdot q^{m'}$ for all $1<i<p^{m+1}$, and so ${\alpha \choose i}\cdot (a+b)=0$ which imply ${\alpha \choose i}\cdot (a+b)^{i}=0.$  
  By Lemma \ref{21},  \[(a+b)^{\circ (\alpha )}=\sum_{i=1}^{\alpha }{\alpha \choose i}(a+b)^{i}=\alpha \cdot a\] because  $(a+b)^{i}=0$ for all $i\geq p^{m}$ (since $p^{m}>n$).
   Therefore, and by Lemma \ref{21},  
\[a\cdot b=(\alpha \cdot a)\cdot b=(a+b)^{\circ( \alpha )}\cdot b=\sum_{i=1}^{\alpha }{\alpha \choose i}\cdot e_{i}(a+b,b)=e_{\alpha }(a+b,b)=0\] because ${\alpha \choose i}$ is divisible by $q^{m'},$ for $1\leq i<\alpha $, and $e_{\alpha }(a+b,b)=0$ since $\alpha >n.$ 

 Observe that $A=\sum_{i=1}^{n}A_{i}$ where $A_{i}$ are Sylow's subgroups of the additive group of $A$. Recall that every $A_{i}$ is a brace (see for example Lemma $17$ in \cite{smok7}).
  Let $a\in A_{i}$ and $b\in A_{j}$ for some $i\neq j$. Since $(a+b)^{m}=0$ we get $a\cdot b=b\cdot a=0$, so $A_{i}\cdot A_{j}=A_{j}\cdot A_{i}=0$. It follows that $A^{\circ }$ is the direct product of groups $A_{i}^{\circ }$ for $i=1, \ldots ,n$. Notice that every $A_{i}^{\circ }$ is a $p$-group; it follows that $A^{\circ }$ is a nilpotent group. On the other hand, if $A$ is nilpotent then $A^{m}=0$ for some $m$ by Theorem $1$  in \cite{smok7}.

\section{Nilpotent braces}
 In this section we will investigate the structure of left braces satisfying special conditions. For the following result we use a short proof which was provided by Ferran Ced{\' o}, which is much better than the original proof from the previous version of this manuscript, and which in addition allows the removal of the assumption that $A$ is a brace of a finite multipermutation level.
 
In this section we will show that if $A$ is a left brace whose additive group has no elements of order two, and moreover  
$(-a)\cdot b= - (a\cdot b)$ for every $a,b\in A$, then $A$ is a
two-sided brace.

{\bf Proof of Theorem \ref{sardegna}.} (Provided by Ferran Ced{\' o} \cite{ferran}.)
Let $a,b,c\in A$. We have that
$ c\circ (-a)\circ b= (2c- c\circ a)\circ b$
and
\begin{eqnarray*}
c\circ (-a)\circ b &=& c\circ ((-a)\cdot b - a + b)\\
&=& c\circ (- (a\cdot b) - a + b)\\
&=& c\circ (- a\circ b +a+b - a + b)\\
&=& c\circ (- a\circ b + 2b)\\
&=& c\circ (2b) - c\circ a\circ b + c\\
&=& 2 (c\circ b) - c\circ a\circ b\\
\end{eqnarray*}
Hence
$ (2c- c\circ a)\circ b= 2 (c\circ b) - c\circ a\circ b.$
In particular, for $a=c^{-1}$ we have that
\[ (2c)\circ b= 2 (c\circ b) -  b.\]
Hence for every $a,b,c\in A$, we have that
$$ (2c- c\circ a)\circ b= (2c)\circ b + b - c\circ a\circ b.$$
Let $x,y,z\in A$. Now we have 
\begin{eqnarray*}
(2x - 2y)\circ z &=& (2x - x\circ x^{-1}\circ (2y))\circ z\\
&=& (2x)\circ z + z - x\circ x^{-1}\circ (2y)\circ z\\
&=& (2x)\circ z + z - (2y)\circ z.
\end{eqnarray*}
Hence
\[ 2((x-y)\circ z)-z=2(x\circ z)-z+z-2(y\circ z)+z=2(x\circ z-y\circ z)+z.\]
 Thus \[2((x-y)\circ z)=2(x\circ z- y\circ z+z).\]
 Since the additive group of $A$ has no elements of order two, we have that 
\[(x-y)\circ z=x\circ z-y\circ z+z.\]
 Note that if $t=x-y,$ then $(t+y)\circ z+z=t\circ z+y\circ z$. Therefore, $A$ is a two-sided brace. The result follows. 

\section{Semidirect product and wreath product}

 In this section we will investigate semidirect product and wreath product of left braces. We will use the notation from Section $6$ in  \cite{cjo}.
 
{\bf Definition $1.$} Let $G$ and $H$  be two left braces. A map $f:G\rightarrow H$ is a homomorphism of left braces if \[f(a+b)=f(a)+f(b), f(a\circ b)=f(a)\circ f(b).\] 

{\bf Definition $2.$}

 Let $N, H$ be left braces, let $\sigma :H\rightarrow Aut(N)$ be a homomorphism of groups from the adjoint group $H^{\circ }$ of $H$ to the group of authomorphisms of the left brace $N$ (see Definition $1$).
 Define the left brace $N\rtimes H$ as follows:
\[(g_{1}, h_{1})+(g_{2}, h_{2})=(g_{1}+g_{2}, h_{1}+h_{2}).\]
\[(g_{1}, h_{1})\circ (g_{2}, h_{2})=(g_{1}\circ \sigma (h_{1})(g_{2}), h_{1}\circ h_{2}).\]

 \begin{lemma}\label{semi} 
 Let $H, N$ be left braces and let $N\rtimes H$ be the semidirect product of braces $H$ and $N$ constructed via $\sigma $.
 The brace $N\rtimes H$ has a finite multipermutation level if and only if braces $H$ and $N$ have a finite multipermutation level. 
\end{lemma}
\begin{proof} Let $A$ be a left brace. 
 It can be observed that $A$ is a brace of a finite multipermutation level if and only if 
 $A^{(n)}=0$ for some $n$ (it follows from results from \cite{rump} in the section about the  socle of a brace, see also \cite{cjo1} for translation to the left braces).
 Let $g_{i}\in N$, $h_{i}\in H$ and $(g_{i},h_{i})\in N\rtimes H$ then 
\[(\cdots (((g_{1}, h_{1})\cdot (g_{2}, h_{2}))\cdot (g_{3}, h_{3}))\cdot \ldots )\cdot  (g_{n}, h_{n})=(d, (\cdots ((h_{1}\cdot h_{2})\cdot h_{3})\cdots) )\]
 for some $d\in N$, where $(g_{i}, h_{i})\cdot (g_{j}, h_{j})=(g_{i}, h_{i})\circ (g_{j}, h_{j})-(g_{i}, h_{i})-(g_{j}, h_{j})$. 

Therefore, if 
brace $A=N\rtimes H$ has a finite multipermutation level then $A^{(n)}=0$ for some $n$, and hence 
 $H^{(n)}=0$. Notice that if $h_{1}=1=0$ then 
\[(g_{1}, 1)\circ (g_{2}, h_{2})=(g_{1}\circ \sigma (1)(g_{2}), h_{2})=(g_{1}\circ g_{2}, h_{2})\]
( since $\sigma $ is a  homomorphism of groups so $\sigma (1)=1$). It follows that  
\[(g_{1}, 1)\cdot (g_{2}, h_{2})=(g_{1}, 1)\circ (g_{2}, h_{2})-(g_{1}, 1)-(g_{2}, h_{2})=(g_{1}\cdot g_{2}, 1).\]

Recall that $A=N\rtimes H$. Therefore, if 
 $A^{(n)}=0$ then 
\[(\cdots (((g_{1}, 1)\cdot (g_{2}, h_{2}))\cdot (g_{3}, h_{3}))\cdots )\cdot  (g_{n}, h_{n})=((\cdots ((g_{1}\cdot g_{2})\cdot g_{3})\cdots )\cdot g_{n}, 1)\]
  therefore $N^{(n)}=0$ and so
 $N$ has a finite multipermutation level. We have shown that if $N\rtimes H$ has a finite multipermutation level then $N$ and $H$ are braces of a finite multipermutation level.
 
Assume now that $H$ has a finite multipermutaton level, so $H^{(m)}=0$ for some $m$.
 Then \[(\cdots (((g_{1}, h_{1})\cdot (g_{2}, h_{2}))\cdot (g_{3}, h_{3}))\cdots )\cdot  (g_{m}, h_{m})=(d,1)\] for some $d\in N$.
Assume that $N$ has a finite multipermutation level, so
$N^{(m')}=0$ for some $m'$.
 Then  
\[(\cdots (((d, {1})\cdot (g_{m+1}, h_{m+1}))\cdot (g_{+2}, h_{m+2}))\cdots )\cdot  (g_{m+m'}, h_{m+m'})=(1,1)=(0,0).\]
 Therefore,  
\[(\cdots (((g_{1}, h_{1})\cdot (g_{2}, h_{2}))\cdot (g_{3}, h_{3}))\cdots  )\cdot  (g_{m+m'}, h_{m+m'})=(1,1)=(0,0),\]
 and hence $A^{(m+m')}=0$ and so $N\rtimes H$ is a brace of finite multipermutation level.
\end{proof}
 We follow definition from \cite{cjo} for the wreath product of left braces.

{\bf  Definition $3$.} Let $G, H$ be two left braces. Then the wreath product $G\wr H$ of $G$ by $H$ is the semidirect product $W\rtimes H$, where \[W=\{f: H\rightarrow G:| \{h\in H:f(h)\neq 1\}|<\infty \}\] and the action of $H$ of $W$ is given by the homomorphism $\sigma : H\rightarrow Aut (W)$ defined by 
 $\sigma (h)(f)(x)=f(hx)$, for all $x\in H$ and $f\in W$.

Recall that the multiplication on $W$ is defined as $(f_{1}\circ f_{2})(x)=f_{1}(x)\circ f_{2}(x)$ and the sum is defined 
$(f_{1}+f_{2})(x)=f_{1}(x)+f_{2}(x)$  for $f_{1}, f_{2}\in W$ and $x\in H$. 

\begin{lemma}  Let $G, H$ be left braces. Then the wreath product $G\wr H$ of $G$ by $H$ is
 a left brace of a finite multipermutation level if and only if $G$ and $H$ are braces of a finite multipermutation level.
\end{lemma}
\begin{proof}  By Lemma \ref{semi},  it suffices to show that the left brace $W$ defined as in Definition $3$ is of a finite multipermutation level  if and only if  $G$ is of a finite multipermutation level.
 Notice that if $f_{1}, f_{2}\in W$ and $x\in H,$ then $(f_{1}\cdot f_{2})(x)=(f_{1}\circ f_{2}-f_{1}-f_{2})(x)=(f_{1}\circ f_{2})(x)-f_{1}(x)-f_{2}(x)=f(x)\cdot f_{2}(x)$. 
 Suppose that $G$ is of a finite multipermutation level. That means that $G^{(m)}=0$ for some $m$.
 For $f_{1}, f_{2}, \ldots , f_{m}\in W$ and $x\in H$ we get 
 \[0=(\cdots (((f_{1}(x)\cdot f_{2}(x))\cdot f_{3}(x))\cdots) f_{m}(x)=((\cdots ((f_{1}\cdot f_{2})\cdot f_{3})\cdots )\cdot f_{m})(x).\]
  Hence $(\cdots ((f_{1}\cdot f_{2})\cdot f_{3})\cdots )f_{m}=0$, and thus $W^{(m)}=0.$ Therefore $W$ is of a finite multipermunation level.

 Conversely, assume that $W$ is of a finite multipermutation level. That means that $W^{(m)}=0$ for some $m$. Let $g_{1}, \ldots , g_{m}\in G$. Let $f_{g_{i}}$ be the element of $W$ defined by $f_{g_{i}}(x)=g_{i}$ if $x=1$ and $f_{g_{i}}(x)=1$ otherwise. 
 Then 
\[(\cdots ((g_{1}\cdot g_{2})\cdot g_{3})\cdots )\cdot g_{m}-((\cdots ((f_{g_{1}}\cdot f_{g_{2}})\cdot f_{g_{3}})\cdots )\cdot f_{g_{m}})(1)=0.\] 
Hence $G^{(m)}=0$ and the result follows.
\end{proof}
{\small\bf Acknowledgements}
{\small  I am very grateful to David Bachiller, Ferran Ced{\' o}, Tatiana Gateva-Ivanova, Jan Okni{\' n}ski and Leandro Vendramin for many helpful suggestions which have improved the previous version of this paper. 
 I am especially grateful to David Bachiller for his suggestion that the assumption that the cardinality of $A$ is an odd number was not necessary in the previous version of Corollary \ref{momik}, and to Ferran Ced{\' o} for  providing a shorter and better  proof of Theorem \ref{sardegna}, which in addition allows the removal of the assumption that the solution associated to $A$ is a multipermutation solution. I am also grateful to Leandro Vendarmin for providing a list of braces which small cardinalities. 
I would also like to thank Michael West for  reviewing  the English-language aspects of this paper. I am very grateful to the unknown referee for their many helpful comments and suggestions which have improved the paper.
  This research was supported with ERC grant 320974, and I would like to thank them for their support. }

\end{document}